\documentclass[11pt,a4paper]{article}
\usepackage{stmaryrd}
\usepackage{amsmath,cases}
\usepackage{amssymb,amsfonts,bm}
\usepackage{amsthm}
\usepackage{dsfont}
\usepackage{color,graphicx,caption,subcaption}
\usepackage[svgnames]{xcolor}
\usepackage{hyperref}
\usepackage{xspace}
\usepackage{fullpage}
\usepackage{placeins}
\usepackage{enumitem}
\usepackage{epstopdf,caption,graphicx,subcaption,mathtools}
\usepackage{url}
\usepackage{verbatim}
\usepackage{amsthm}
\usepackage{commath}
\usepackage{multirow}
\usepackage[ruled]{algorithm2e}
\usepackage{ulem}
\usepackage{booktabs}
\usepackage{float}
\usepackage{xpatch}
\xpatchcmd{\paragraph}{\normalfont}{{\normalfont\bfseries}}{}{}
\usepackage{tabularx}

\usepackage{epsfig,epstopdf,bm}
\usepackage{graphicx}

\newcommand{\be}{\begin{equation}}
	\newcommand{\ee}{\end{equation}}
\newcommand{\ba}{\begin{array}}
	\newcommand{\ea}{\end{array}}
\def\bea{\begin{eqnarray}}
	\def\eea{\end{eqnarray}}
\def \beas{\begin{eqnarray*}}
	\def \eeas{\end{eqnarray*}}
\newtheorem{definition}{Definition}[section]

\newtheorem{remark}{Remark}[section]

\newtheorem{thm}{Theorem}[section]
\newtheorem{lemma}{Lemma}[section]

\DeclareMathOperator*{\argmin}{arg\,min}

\def\<{\langle}
\def\>{\rangle}

\def\ba{{\bf a}}

\begin{document}
\begin{center}
	\large \bf Optimal local linear convergence of Nesterov’s accelerated gradient method for 
	$C^2$ functions under the Polyak--Łojasiewicz inequality
\end{center}
\vspace{1pt}
\begin{center}
	\large  Zixu Feng$^1$ \quad and\quad Hao Yuan$^2$
\end{center} 
\vspace{1pt}	
{\small	
\begin{center}\it
	\textsuperscript{1}{School of mathematical science, Chengdu University of Technology, Chengdu 610064,  P. R. China}\\
 \textsuperscript{2}{School of Mathematics, Sichuan University, Chengdu 610064, P.R. China}
\end{center}
}
\normalsize
\rm
\begin{center}\it		
	E-mail: zixu\_feng123@163.com, \quad drifteryuan@163.com
\end{center}

\footnotesize
\vspace{6pt}
\textbf{Abstract.} In this work, we establish that Nesterov's accelerated gradient method, applied to $C^2$ functions satisfying the Polyak--\L ojasiewicz inequality around local minimizers, achieves the optimal local linear convergence rate $\rho=\frac{\sqrt{3L+\mu}-2\sqrt{\mu}}{\sqrt{3L+\mu}}+\varepsilon$, where $\varepsilon$ is an arbitrarily small constant. Our analysis requires neither higher-order smoothness beyond $C^2$ of the objective function nor any additional geometric regularity of the submanifold of local minimizers. The key novelty lies in a two-stage argument: we first establish a coarse yet valid local linear convergence rate and then, building upon this a priori convergence guarantee, obtain a refined characterization of the linearized iteration operator, which yields the optimal rate. As a result, we only need to slightly strengthen the standard $C^{1,1}$ assumption, which is commonly required in theoretical analyses of linear convergence for first-order methods, to $C^2$ smoothness. Moreover,  the same analytical framework allows us to recover, under identical conditions, the optimal local exponential convergence rate $\sqrt{\mu}$ for the continuous-time Heavy Ball dynamics. Finally, a representative numerical experiment corroborates our theoretical findings.
\vspace{6pt}

{\bf Keywords:} Nesterov's accelerated gradient method, Polyak--Łojasiewicz inequality, optimal local linear convergence rate

\vspace{6pt}
{\bf MSC codes.} 65K05, 90C26, 90C30 
\normalsize



\section{Introduction}
\setcounter{equation}{0}
We consider the following unconstrained minimization problem
\begin{align}\label{U-min-p}
	\text{Find}\quad x_{*}\in\argmin_{x\in\mathbb{R}^d} f(x),
\end{align}
where $f:\mathbb{R}^d\to\mathbb{R}$ is lower-bounded and twice continuously differentiable. For large-scale optimization, due to the relatively low computational cost per iteration, first-order algorithms are widely used . Among them Nesterov's accelerated gradient (NAG) methods \cite{2018Nestreov} have attracted considerable attention for their potential to accelerate convergence toward local minimizers. Specifically, given two initial points $x^0, x^1 \in \mathbb{R}^d$, the NAG iterates are defined for $n \geq 1$ by  
\begin{align}\label{NAG}
	\left\{
	\begin{aligned}
		&y^{n} = x^n + \beta v_n,\; v_n:=x^n-x^{n-1} \\
		&x^{n+1} = y^n - \alpha \nabla f(y^n)
	\end{aligned}
	\right.\;,
\end{align}
where $\alpha > 0$ is the stepsize and $\beta \in [0,1)$ is the momentum parameter.

In the locally strongly convex setting, i.e., the Hessian of $f$ at a local minimizer $x_{*}$ satisfies the second-order sufficient condition ($\nabla^2 f(x_{*}) \succ 0$), it is well known that NAG can achieve an accelerated local convergence rate compared to standard gradient descent method (GD). Indeed, building on the analysis in \cite{lessard2016}, one can show that for optimally tuned hyperparameters $\alpha$ and $\beta$, and under the additional assumption that the iterates are initialized sufficiently close to $x_{*}$, the method converges linearly with asymptotic rate  
\[
\rho_{\mathrm{NAG}} = \frac{\sqrt{3L + \mu} - 2\sqrt{\mu}}{\sqrt{3L + \mu}} + \varepsilon,
\]
where $L$ and $\mu$ denote the largest and smallest eigenvalues of $\nabla^2 f(x_{*})$, respectively, and $\varepsilon > 0$ is an arbitrarily small constant.
In contrast, GD with optimal stepsize achieves only the optimal local linear rate  
\[
\rho_{\mathrm{GD}} = \frac{L - \mu}{L + \mu} + \varepsilon.
\]
As shown in the inequality  
\[
\frac{\sqrt{3L + \mu} - 2\sqrt{\mu}}{\sqrt{3L + \mu}} \leq \frac{L - \mu}{L + \mu}, \quad \text{for all } L \geq \mu > 0,
\]
with strict inequality when $L > \mu$, the NAG indeed provides a provably faster local convergence rate in this structured.

In many practical applications, the strong convexity assumption is frequently violated. A wide range of optimization problems of interest, such as those arising in machine learning, signal processing, and inverse problems, often exhibit non-isolated minimizers and may even admit infinitely many minimizers. This has motivated decades of research into the convergence behavior of momentum-based algorithms, particularly the Heavy Ball method and NAG, under assumptions weaker than strong convexity. To this end, several classical relaxed conditions have been proposed. Among the most prominent is the Polyak--\L{}ojasiewicz (PL) inequality, independently introduced by Polyak~\cite{1963Polyak} and \L{}ojasiewicz~\cite{1963LS}. Owing to its simplicity, broad applicability, and validity even for certain non-convex functions, the PL inequality has become a widely used tool for establishing linear convergence of gradient-based algorithms without requiring strong convexity. Beyond the PL condition, other relaxations of strong convexity have been studied, including the error bound property~\cite{1993Luo}, quadratic growth~\cite{2000Anitescu}, the restricted secant inequality~\cite{2013Zhang}, and essential strong convexity~\cite{2014Liu}. In~\cite{2019Necoara}, the authors further discuss a family of closely related conditions, such as quasi-(weak) strong convexity, quadratic functional growth, quadratic gradient growth, and quadratic under-approximation. Under the assumption that $f$ has a Lipschitz-continuous gradient, these conditions often imply one another (see~\cite[Theorem~4]{2019Necoara} and~\cite[Theorem~2]{2016Karimi}). Moreover, if one further assumes that $f \in C^2$, then according to the recent results in~\cite{2025Fast} and~\cite[Lemma 3.2]{2026HBOpt}, all these conditions are equivalent in a sufficiently small neighborhood of minimizers. 

A central question regarding the aforementioned weakened strong convexity conditions is whether momentum-based algorithms, such as the Heavy Ball method and NAG, can achieve acceleration over standard gradient descent. Here, by ``acceleration'' we mean that the linear convergence rate $\rho$ improves from depending on $\frac{\mu}{L}$ to depending on $\sqrt{\frac{\mu}{L}}$.
To the best of our knowledge, under only the assumptions that $f$ satisfies one of the aforementioned conditions and that $\nabla f$ is $L$-Lipschitz continuous or $f\in C^{1,1}$, no general theoretical result guarantees such acceleration for momentum methods. In fact, negative results exist: \cite{2023OnLB} shows that for general PL functions with $L$-Lipschitz gradients, momentum-based algorithms cannot achieve global acceleration. Nevertheless, acceleration can be recovered under additional structural assumptions on $f$. All known positive results assume that at least $\nabla f$ is $L$-Lipschitz continuous or that $f\in C^{1,1}$.  For the Heavy Ball method, \cite{2022Wang} establishes provable acceleration for a class of non-convex PL functions whose non-convexity is ``averaged out'' in a suitable sense. More recently, \cite{2026HBOpt} shows that if $f \in C^4$ satisfies the PL inequality around its minimizer set, then the Heavy Ball method recovers the optimal asymptotic local convergence rate originally derived by Polyak~\cite{1964P} for $C^2$ strongly convex functions, thereby achieving local acceleration. For NAG, \cite{2019Necoara} proves acceleration of NAG for quasi-strongly convex functions, provided that the auxiliary iterates share the same projection onto the set of minimizers. \cite{2024HBM} demonstrates acceleration of NAG for convex functions satisfying the quadratic growth condition. \cite{2025SqC} establishes acceleration of NAG for $\gamma$-strongly quasar-convex functions (which reduce to quasi-strongly convex when $\gamma = 1$) under an additional curvature condition.
\cite{2025Gupta} proves acceleration of NAG for quasi-strongly convex functions under the assumption that the set of minimizers is a $C^2$ embedded submanifold of $\mathbb{R}^d$ and that the projection mapping onto this set satisfies a mild regularity condition. More recently,
\cite{2026HJ} clarifies provable acceleration by momentum for PL functions under an aiming condition. 

\vspace{0.3cm}

\textbf{Our contribution.} In this work, we prove that NAG achieves local acceleration under the assumption that $f$ satisfies the PL inequality in a neighborhood of its minimizer set and that $f \in C^2$. Specially, we recover the optimal asymptotic local convergence rate
\[
\rho = \frac{\sqrt{3L + \mu} - 2\sqrt{\mu}}{\sqrt{3L + \mu}} + \varepsilon,
\]
where $\varepsilon > 0$ is an arbitrarily small constant, which matches the sharp rate established by~\cite{lessard2016}. To the best of our knowledge, this is the first result establishing such a sharp local acceleration rate for NAG under merely $C^2$ smoothness and a local PL condition. Our analysis relies on a novel use of the a priori linear convergence rate, which enables a refined spectral characterization without requiring higher-order smoothness of $f$ or additional geometric regularity of the local minimizer submanifold. Moreover, leveraging this idea, we also establish the optimal local exponential decay rate $\sqrt{\mu}$ for the continuous-time Heavy Ball accelerated flow trajectory associated with NAG.

\vspace{0.3cm}

The remainder of this paper is organized as follows: Section \ref{Sec2} introduces the necessary notations and assumptions, recalls key properties of the objective function under these assumptions, and presents the main results. Section \ref{Sec3} first establishes several requisite lemmas before providing the proofs of the main results. In Section \ref{Sec4}, we validate the theoretical findings through comprehensive numerical experiments. Finally, conclusions are drawn in Section \ref{Sec5}.

\section{Preliminaries and main results}\label{Sec2}
In this section, we first introduce problem settings, basic notations, and some important properties of the problem. Then, the main results are presented.

\subsection{Problem settings, notations, and properties}
The goal of this paper is to analyze the local convergence rate of iterative algorithms near minimizers. 
To this end, in addition to assuming that the objective function $f$ in~\eqref{U-min-p} is bounded from below and of class $C^2$, 
we introduce the following local PL condition.

\begin{definition}\label{Local-conditions}
Let $f : \mathbb{R}^d \to \mathbb{R}$ be a $C^2$ function, and let $x_* \in \mathbb{R}^d$ be a minimizer of $f$. 
We say that $f$ satisfies the local PL condition at $x_*$ if there exist a sufficiently small open neighborhood $U \subset \mathbb{R}^d$ of $x_*$, and a constant $\nu > 0$ such that
\[
\|\nabla f(x)\|^2 \ge 2\nu \bigl( f(x) - f(x_*) \bigr), \quad \text{for all }\; x \in U,
\]
where $\|\cdot\|$ denotes the Euclidean norm.
\end{definition}
Note that the $L$-Lipschitz continuity of the gradient is not assumed explicitly here, as it is automatically satisfied on the bounded set $U$ for a $C^2$ function.

In particular, it is worth noting that for a $C^2$ objective function, if it is coercive or its set of minimizers is bounded, then recent results~\cite{2025Criscitiello,2025Nejma} show that the global PL condition guarantees uniqueness of the minimizer. Thus, the global PL condition is, in some sense, a rather strong assumption.

Next, we introduce the following set and associated function value level constructed from a local minimizer $x_*$:
\begin{align}\label{mathcal_S}
    \mathcal{S} := \left\{ x \in \mathbb{R}^d \,\Big|\, x \text{ is a local minimizer of } f \text{ and } f(x) = f(x_*) \right\}.
\end{align}

In the work of~\cite{2025Fast}, it was shown that the local PL condition defined in~\eqref{Local-conditions} is equivalent to the following local Morse--Bott condition.

\begin{definition}\label{MB}
Let $f : \mathbb{R}^d \to \mathbb{R}$ be a $C^2$ function. 
We say that $f$ satisfies the local Morse--Bott condition near a local minimizer $x_*$ if there exists a sufficiently small open neighborhood $U \subset \mathbb{R}^d$ of $x_*$ such that the set
\[
    \mathcal{S}_U := \mathcal{S} \cap U
\]
is a $C^1$ embedded submanifold of $\mathbb{R}^d$, and for every $x_*' \in \mathcal{S}_U$,
\[
    K_{x_*'} := \ker\big( \nabla^2 f(x_*') \big) 
    = \left\{ v \in \mathbb{R}^d \,\big|\, \nabla^2 f(x_*') [v] = 0 \right\}
    = T_{x_*'} \mathcal{S}_U.
\]
Here, the tangent space $T_{x_*'} \mathcal{S}_U$ is defined by
\[
    T_{x_*'} \mathcal{S}_U 
    := \left\{ \gamma'(0) \in \mathbb{R}^d \,\Big|\, \gamma : (-\varepsilon, \varepsilon) \to \mathcal{S}_U \text{ is a } C^1 \text{ curve with } \gamma(0) = x_*' \right\}.
\]
\end{definition}

\begin{remark}\label{rem:nontrivial}
If $\dim T_{x_*'} \mathcal{S}_U = d$ for some (hence all) $x_*' \in \mathcal{S}_U$, then $\mathcal{S}_U$ is an open subset of $\mathbb{R}^d$. In this case, the condition $f(x_*') = f(x_*)$ for all $x_*' \in \mathcal{S}_U$ implies that $f$ is constant on a neighborhood of $x_*$. We exclude this trivial scenario throughout the paper, as it leads to a completely flat landscape. Consequently, we always assume that $\mathcal{S}_U$ is a proper submanifold of $\mathbb{R}^d$, i.e., $\dim T_{x_*'} \mathcal{S}_U < d$.
\end{remark}

In addition, we introduce the normal space
\[
    N_{x_*'}\mathcal{S}_U := \bigl(T_{x_*'}\mathcal{S}_U\bigr)^\perp,
\]
and let $P_{x_*'} : \mathbb{R}^d \to N_{x_*'}\mathcal{S}_U$ denote the orthogonal projection onto this normal space. In the nontrivial case, i.e., when $\mathcal{S}_U$ is not an open subset of $\mathbb{R}^d$, the local Morse--Bott condition implies that for every $x_*' \in \mathcal{S}_U$, the Hessian $\nabla^2 f(x_*')$ is nondegenerate on $N_{x_*'}\mathcal{S}_U$. Consequently, its eigenvalues restricted to $N_{x_*'}\mathcal{S}_U$ satisfy
\[
    0 < \mu(x_*') := \lambda_{\min}\bigl(\nabla^2 f(x_*')|_{N_{x_*'}\mathcal{S}_U}\bigr)
    \leq 
    \lambda_{\max}\bigl(\nabla^2 f(x_*')|_{N_{x_*'}\mathcal{S}_U}\bigr) =: L(x_*').
\]
Note that $\lambda_{\max}\bigl(\nabla^2 f(x_*')|_{N_{x_*'}\mathcal{S}_U}\bigr)$ is precisely the largest eigenvalue of the full Hessian $\nabla^2 f(x_*')$ on $\mathbb{R}^d$. Indeed, under the local Morse--Bott condition at a local minimizer $x_*'\in\mathcal{S}_{U}$, the Hessian vanishes on the tangent space $T_{x_*'}\mathcal{S}_U$, and $\mathbb{R}^d$ decomposes orthogonally as $T_{x_*'}\mathcal{S}_U \oplus N_{x_*'}\mathcal{S}_U$. Hence all nonzero eigenvalues of $\nabla^2 f(x_*')$ are contained in the normal subspace, and the maximum eigenvalue is attained there. 

We also emphasize that the local constants $\mu(x_*')$ and $L(x_*')$ generally depend on the base point $x_*' \in \mathcal{S}_U$. However, our analysis is purely local. Since $f$ is of class $C^2$ and $\mathcal{S}_U$ is a $C^1$ embedded submanifold, the tangent spaces $T_{x_*'}\mathcal{S}_U$, and hence the normal spaces $N_{x_*'}\mathcal{S}_U$, vary continuously with $x_*'$. It follows that the orthogonal projections $P_{x_*'}$ depend continuously on $x_*'$, and so do the extremal eigenvalues 
$\mu(x_*'), L(x_*')$ by the continuity of the map $x_*' \mapsto P_{x_*'} \nabla^2 f(x_*') P_{x_*'}$.
Therefore, on any sufficiently small neighborhood of the reference point $x_*$, the quantities $\mu(x_*')$ and $L(x_*')$ remain uniformly bounded away from zero and infinity. This mild dependence on the base point can be absorbed into the small parameter $\varepsilon > 0$, which governs both the stepsize restriction and the asymptotic convergence rate. For notational simplicity, we henceforth suppress the explicit dependence on $x_*'$ and denote these uniform local bounds simply by $\mu$ and $L$, which, up to the negligible perturbation governed by $\varepsilon$, coincide with the standard PL constant $\mu$ and Lipschitz gradient constant $L$ in the usual sense.

Finally, for notational convenience, we adopt the following conventions throughout the paper:
\begin{itemize}
    \item The symbol $C > 0$ denotes a generic positive constant that is independent of the iterate $x$ and may vary from line to line. 
    \item We write $a \approx b$ for two nonnegative quantities $a, b$ if there exist constants $C_1, C_2 > 0$, independent of $x$, such that
    \[
        C_1 b \le a \le C_2 b.
    \]
\end{itemize}
\subsection{Main results}
In this subsection, under the local PL condition, we establish a sharp local linear convergence rate for NAG around local minimizers. Despite the failure of the second-order sufficient conditions due to degeneracy of the Hessian, we show that NAG achieves the optimal local contraction factor that matches the optimal one for strongly convex quadratic problems (see~\cite{lessard2016}) up to an arbitrarily small $\varepsilon > 0$. This result highlights the robustness of momentum-based acceleration in nonconvex optimization when the local PL condition is properly exploited.

We consider the local dynamics of NAG near local minimizers $ x_*$. To this end, we introduce the following Lyapunov function : for all $x,v\in \mathbb{R}^d$, 
\begin{align}\label{Lyapunov-function}
	\mathcal{L}(x, v) := f(x) - f(x_*) + \frac{\beta}{2\alpha} \|v\|^2 - \frac{\beta^2}{2(1+\beta)} \nabla^2f(x_*)[v,v],
\end{align}
where $\nabla^2f(x_*)[v,v]:=\big\langle \nabla^2f(x_*)[v], v\big\rangle$. The following result establishes local Lyapunov stability for this sequence.

\begin{thm}\label{Exp-Lyapunov-Stab}
	Let  $f$  satisfy the local PL condition at the local minimizer $x_*\in\mathcal{S}$. Then, for every sufficiently small  $ \varepsilon > 0 $, there exists a sufficiently small open neighborhood $U \subset \mathbb{R}^d$ of $x_*$ such that for any $x^0, x^1 \in U$, the discrete Lyapunov sequence generated by NAG satisfies
	\[
	0\le\mathcal{L}(x^{n+1}, v_{n+1}) \leq \rho_{\varepsilon,\alpha,\beta} \, \mathcal{L}(x^n, v_n)
	\quad \text{for all } n \geq 1,
	\]
	provided that the step size  $ \alpha $ satisfies
	\[
	\alpha \in \left( 0, \frac{2(\beta+1)}{(L+\varepsilon)(2\beta+1)} \right), \quad\text{for all}\; \beta \in [0,1),
	\]
	where  $ \rho_{\varepsilon,\alpha,\beta} \in (0,1) $  is a contraction factor depending on  $\varepsilon$, $ \alpha $, and  $ \beta $.
\end{thm}
The theorem above implies the local linear convergence of NAG, yet the contraction factor cannot be precisely quantified.
Next, for NAG, the optimal local convergence rate is obtained in the following theorem. 

\begin{thm}\label{Opt-convergence}
	Let $f$ satisfy the local PL condition at the local minimizer $x_*\in\mathcal{S}$. Then, for every sufficiently small $\varepsilon > 0$, there exists a sufficiently small open neighborhood $U \subset \mathbb{R}^d$ of $x_*$ such that for any $x^0, x^1 \in U$, the sequence $\{x^n\}_{n\in\mathbb{N}}$ generated by NAG converges linearly to a local minimizer $x_*'\in\mathcal{S}$, i.e.,
	\begin{align*}
		\|x^n-x_*'\|\le C_{\varepsilon}\big(\| x^{1} - x_*' \|+\| x^{0} - x_*' \|\big)(\rho_{\alpha,\beta}+\varepsilon)^n,\quad\forall\; n\ge1,
	\end{align*}
	for the step size  $\alpha$  satisfying
	\begin{align*}
		\alpha \in \left( 0, \frac{2(\beta+1)}{(L+\varepsilon)(2\beta+1)} \right), \quad\text{for all}\; \beta \in [0,1),
	\end{align*}
	where $\rho_{\alpha,\beta}\in(0,1)$ is a contraction factor depending on  $ \alpha $ and  $ \beta $, and $C_{\varepsilon}>0$ is a constant depending on $\varepsilon$. In particular, with the optimal parameter choice
	\begin{align*}
		\alpha_{{\rm opt}}=\frac{4}{3L+\mu}\quad\text{and}\quad \beta_{{\rm opt}}=\frac{\sqrt{3L+\mu}-2\sqrt{\mu}}{\sqrt{3L+\mu}+2\sqrt{\mu}},
	\end{align*}
	 NAG achieves the optimal linear convergence rate
	\begin{align}
			\|x^n-x_*'\|\le C_{\varepsilon} \big(\| x^{1} - x_*' \|+\| x^{0} - x_*' \|\big)\left(\frac{\sqrt{3L+\mu}-2\sqrt{\mu}}{\sqrt{3L+\mu}}+\varepsilon\right)^n,\quad\forall\; n\ge1.
	\end{align}
\end{thm}

In essence, the second theorem refines the first by providing a precise quantification of the contraction factor, building directly on the local linear convergence established therein. Notably, this approach extends seamlessly to the Heavy Ball ordinary differential equation, namely
\begin{align}\label{Accelerated-flow}
	\ddot{x}(t) + \gamma\dot{x}(t) + \nabla f(x(t)) = 0,
\end{align}
yielding the optimal exponential local convergence rate $\sqrt{\mu}$ with $\gamma=2\sqrt{\mu}$ for its trajectories under the assumption that $f \in C^2$ and satisfies the local PL condition. 
\begin{thm}\label{Opt-convergence-flow}
	Let $f$ satisfy the local PL condition at the local minimizer $x_*\in\mathcal{S}$. Then, for every sufficiently small $\varepsilon > 0$, there exists a sufficiently small open neighborhood $\mathcal{U} \subset \mathbb{R}^d \times \mathbb{R}^d$ of $(x_*, 0)$ such that for any initial condition $(x(0), \dot{x}(0)) \in \mathcal{U}$, the trajectories $x(t)$ generated by \eqref{Accelerated-flow} converges exponentially to a local minimizer $x_*'\in\mathcal{S}$, i.e.,
	\begin{align*}
		\|x(t)-x_*'\|\le C_{\varepsilon}\big(\|x(0)-x_*'\|+\|\dot{x}(0)\|\big)e^{-(\rho_{\gamma}-\varepsilon)t},\quad\forall\; t\ge0,
	\end{align*}
	where $\rho_{\gamma}=\frac{\gamma - \sqrt{\max\{0,\, \gamma^2 - 4\mu\}}}{2}$ is a convergence rate depending on $\gamma > 0$, and $C_{\varepsilon} > 0$ is a constant depending on $\varepsilon$. In particular, with the optimal damping parameter
	\begin{align*}
		\gamma=2\sqrt{\mu},
	\end{align*}
	the flow \eqref{Accelerated-flow} achieves the optimal cxponential convergence rate
	\begin{align}
		\rho_{{\rm opt}}=\sqrt{\mu}.
	\end{align}
\end{thm}

\section{Proof of main results}\label{Sec3} 
In this section, we provide complete proofs of all the results.

\subsection{Technical lemmas}

Before presenting the proof, we introduce several key lemmas that will be instrumental in establishing various aspects of our results. Firstly, under the local PL condition, we provide the upper and lower bounds for this Lyapunov function \eqref{Lyapunov-function}.
\begin{lemma}\label{Lyapunov-Inf-Sup}
	Let  $ f $  satisfy the local PL condition at the local minimizer $ x_*\in\mathcal{S} $. Then, there exists a sufficiently small open neighborhood $U \subset \mathbb{R}^d$ of $x_*$ such that for all $x \in U $, the Lyapunov function $\mathcal{L}(x,v)$ defined in~\eqref{Lyapunov-function} satisfies
	\begin{align*}
		\mathcal{L}(x,v)\approx \left(\big\|\nabla f(x)\big\|^2+\|v\|^2\right),\quad\forall\; \alpha\in\left(0,\frac{1+\beta}{(L+\varepsilon)\beta}\right).
	\end{align*}
	Here, $A \approx B$ means that there exist constants $c_1, c_2 > 0$  such that
	\[
	c_1 B \leq A \leq c_2 B.
	\]
\end{lemma}
\begin{proof}
	First, since $f$ satisfies the local PL condition at the local minimizer $x_*$, there exists $\sigma_1 > 0$ such that for all $x \in \mathcal{B}_{\sigma_1}(x_*) := \{ y \in \mathbb{R}^d : \|y - x_*\| < \sigma_1 \}$, the PL inequality holds. Together with the non-negativity of the Hessian at $x_*$, i.e., $\nabla^2 f(x_*)[v,v] \ge 0$, we immediately obtain the upper bound
	\[
	\mathcal{L}(x,v)
	\le \frac{1}{2\nu} \big\| \nabla f(x) \big\|^2 
	+ \frac{\beta}{2\alpha} \|v\|^2 
	\le C \left( \big\| \nabla f(x) \big\|^2 + \|v\|^2 \right).
	\]
	
	Next, by the boundedness of the Hessian at $x_*$, i.e., $\nabla^2 f(x_*)[v,v] \le L \|v\|^2$, it follows that
	\[
	\left( \frac{\beta}{2\alpha} - \frac{\beta^2 L}{2(1+\beta)} \right) \|v\|^2 
	\ge C \|v\|^2,\quad \text{whenever}\quad \alpha \in \left( 0, \frac{1+\beta}{L\beta} \right).
	\]
	
	Finally, by~\cite[Lemma~1.4, Proposition~2.2]{2025Fast}, there exists $\sigma_2 > 0$ such that for all $x \in \mathcal{B}_{\sigma_2}(x_*)$, the projection of $x$ onto the set $\mathcal{S}$, denoted $P_{\mathcal{S}}(x)$, is nonempty and the quadratic growth condition holds:
	\[
	f(x) - f(x_*) 
	= f(x) - f(P_{\mathcal{S}}(x)) 
	\ge C \|x - P_{\mathcal{S}}(x)\|^2
	= C\,\mathrm{dist}(x,\mathcal{S})^2.
	\]
	Noting that $\|x - x_*\| \ge \|x - P_{\mathcal{S}}(x)\|$ and combining this with the local Lipschitz continuity of $\nabla f(x)$, there exists a sufficiently small constant $\sigma \in (0, \min\{\sigma_1, \sigma_2\}]$ such that
	\[
	f(x) - f(x_*) \ge C \big\| \nabla f(x) \big\|^2, \quad \forall\; x \in \mathcal{B}_\sigma(x_*).
	\]
	
	Combining the above estimates, we conclude the lower bound
	\[
	\mathcal{L}(x, v) \ge C \left( \big\| \nabla f(x) \big\|^2 + \|v\|^2 \right),
	\]
	which together with the upper bound yields the desired equivalence.
\end{proof}

To prove \textbf{Theorem~\ref{Opt-convergence}}, we adopt the same decomposition strategy as in~\cite{2025Gupta,2026HBOpt}, splitting the full space into tangent and normal components to analyze the convergence rate. We now introduce the matrix $\mathcal{G}_{\alpha}(x_*') : N_{x_*'}\mathcal{S}_{U} \to N_{x_*'}\mathcal{S}_{U}$ defined by
\begin{align*}
	\mathcal{G}_{\alpha}(x_*') 
	:= \big( I - \alpha\nabla^2 f(x_*') \big) P_{x_*'}.
\end{align*}
We also define the block matrix $G_{\alpha,\beta}(x_*')$ acting on the invariant subspace $\mathcal{N}_{x_*'}:=N_{x_*'}\mathcal{S}_{U} \times N_{x_*'}\mathcal{S}_{U} \times T_{x_*'}\mathcal{S}_{U}$ as
\begin{align}\label{G-block-oper}
	G_{\alpha,\beta}(x_*') :=
	\begin{pmatrix}
		(1+\beta)\,\mathcal{G}_{\alpha}(x_*') & -\beta\,\mathcal{G}_{\alpha}(x_*') & \mathbf{0} \\
		I & \mathbf{0} & \mathbf{0} \\
		\mathbf{0} & \mathbf{0} & \beta\, I
	\end{pmatrix}.
\end{align} 
The spectral radius of the matrix $G_{\alpha,\beta}(x_*')|_{\mathcal{N}_{x_*'}}$ is characterized as follows.
 
\begin{lemma}\label{G-spectrum}
	Let $f$ satisfy the local PL condition at the local minimizer $ x_* $. Then the spectral radius of the matrix  $ G_{\alpha,\beta}(x_*')|_{\mathcal{N}_{x_*'}} $  is strictly less than  $ 1 $  if
	\begin{align*}
		\alpha\in\left(0,\frac{2(\beta+1)}{L(2\beta+1)}\right)\quad\text{for all}\;\beta\in[0,1).
	\end{align*}
	Moreover, the spectral radius is minimized over 
	\[
	(\alpha,\beta) \in \left(0,\frac{2(\beta+1)}{L(2\beta+1)}\right) \times [0,1)
	\] 
	 by the optimal choice
	\begin{align*}
	\alpha_{\mathrm{opt}} = \frac{4}{3L + \mu}, \qquad 
	\beta_{\mathrm{opt}} = \frac{\sqrt{3L + \mu} - 2\sqrt{\mu}}{\sqrt{3L + \mu} + 2\sqrt{\mu}},
	\end{align*}
	and the minimal value is given by
	\begin{align*}
	\rho_{\mathrm{opt}}\bigl( G_{\alpha,\beta}(x_*') \bigr) 
	= \frac{\sqrt{3L + \mu} - 2\sqrt{\mu}}{\sqrt{3L + \mu}}.
	\end{align*}
\end{lemma}
\begin{proof}
	It is clear that the spectral radius of  $ G_{\alpha,\beta}(x_*')|_{\mathcal{N}_{x_*'}} $  is determined by the $ \beta $  and the upper-left  $ 2\times 2 $  block matrix  
	\begin{align*}
	\widetilde{G}_{\alpha,\beta}(x_*') :=
	\begin{pmatrix}
		(1+\beta)\mathcal{G}_\alpha(x_*') & -\beta\mathcal{G}_\alpha(x_*') \\
		I & \mathbf{0}
	\end{pmatrix}.
	\end{align*}
	Note that each $ \lambda \in{ \rm Spec}\big(\widetilde{G}_{\alpha,\beta}(x_*')\big) $ satisfies the quadratic equation
	\begin{align}\label{eq:char_eq}
		\lambda^2 - (1+\beta)\mu_{\alpha}\lambda + \mu_{\alpha}\beta = 0,
	\end{align}
	for some  $ \mu_{\alpha} \in{ \rm Spec}\bigl(\mathcal{G}_\alpha(x_*')\bigr) $.
	
	Below, we can discuss what choices of parameter  $ \alpha $ ensure that for all $\beta\in[0,1)$, all eigenvalues  $ \lambda $  lie inside the unit circle. This depends solely on  $ |\lambda| $. Using the quadratic formula, we obtain the following expression 
	\begin{align*}
	|\lambda|=\left\{
	\begin{aligned}
		&\frac{1}{2}(1+\beta)|\mu_{\alpha}|+\frac{1}{2}\sqrt{D}&\quad\text{if}\quad D\ge0,\\
		&\sqrt{\beta|\mu_{\alpha}|}&\quad\text{otherwise},
	\end{aligned}
	\right.
	\end{align*}
	where $D=(1+\beta)^2\mu^2_{\alpha}-4\beta\mu_{\alpha}$. To ensure that all eigenvalues $ \lambda $ of the linearized iteration operator satisfy $ |\lambda| < 1 $ for every $ \mu_\alpha \in {\rm Spec}(\mathcal{G}_\alpha(x_*')) \subseteq [\,1 - \alpha L,\; 1 - \alpha \mu\,] $ and for all $\beta \in [0,1)$, the parameter $ \alpha > 0 $ must satisfy the classical Schur conditions:  
	\begin{align*}
	 1 - \mu_\alpha > 0,\quad 1 + (1 + 2\beta)\mu_\alpha > 0,\quad\text{and}\quad \beta |\mu_\alpha| < 1.
	\end{align*}
	These inequalities are guaranteed uniformly when  
	\begin{align*}
		\alpha\in\left(0,\frac{2(\beta+1)}{L(2\beta+1)}\right)\quad\text{for all}\;\beta\in[0,1).
	\end{align*}
	Finally, according to \cite[Proposition 1]{lessard2016}, the optimal parameters
	\[
	\alpha_{\mathrm{opt}} = \frac{4}{3L + \mu}, \qquad 
	\beta_{\mathrm{opt}} = \frac{\sqrt{3L + \mu} - 2\sqrt{\mu}}{\sqrt{3L + \mu} + 2\sqrt{\mu}},
	\]
	minimize the spectral radius of  $ \widetilde{G}_{\alpha,\beta}(x_*') $ with respect to $\alpha,\beta$, yielding
	\[
	\rho_{\mathrm{opt}}\bigl( \widetilde{G}_{\alpha,\beta}(x_*') \bigr) 
	= \frac{\sqrt{3L + \mu} - 2\sqrt{\mu}}{\sqrt{3L + \mu}}.
	\]
	A direct computation shows that  $ \beta_{\mathrm{opt}} \le \rho_{\mathrm{opt}}\bigl( \widetilde{G}_{\alpha,\beta}(x_*') \bigr) $ with equality only when  $ \mu = L $. 
	Since 
	\[
	{\rm Spec}\bigl(G_{\alpha,\beta}(x_*')|_{\mathcal{N}_{x_*'}}\bigr) 
	= {\rm Spec}\bigl(\widetilde{G}_{\alpha,\beta}(x_*')\bigr) \cup \{\beta\},
	\]
	the spectral radius of ${\rm Spec}\bigl(G_{\alpha,\beta}(x_*')|_{\mathcal{N}_{x_*'}}\bigr) $ is governed by the larger of the two quantities. Hence,
	\[
	\rho_{\mathrm{opt}}\bigl( G_{\alpha,\beta}(x_*')|_{\mathcal{N}_{x_*'}} \bigr) 
	= \frac{\sqrt{3L + \mu} - 2\sqrt{\mu}}{\sqrt{3L + \mu}}.
	\]
\end{proof}

The following is a classical result in functional analysis that relates the local convergence rate of a nonlinear iteration to the spectral radius of its linear part. It is commonly used in the stability analysis of ordinary differential equations and dynamical systems.
\begin{lemma}\label{Nonlinear-Iteration}
	Suppose that the linear operator $T$ on a Hilbert space $X$ satisfies the condition $\rho(T)=\rho<1$, and the sequence $\big\{v^n\big\}_{n\in\mathbb{N}}\subset X$ satisfies:
	\begin{align*}
		v^{n+1}=Tv^n+Y(v^n)\quad and \quad \lim\limits_{\|v\|_{X}\to0}\frac{\|Y(v)\|_{X}}{\|v\|_{X}}=0.
	\end{align*}
	Then, for all sufficiently small $\varepsilon$, there exists $\sigma$ such that for all $\|v^0\|_X\le\sigma$, 
	\begin{align*}
		\|v^n\|_{X}\le C_{\varepsilon}\|v^0\|_X(\rho+\varepsilon)^n.
	\end{align*}
\end{lemma}
\begin{proof}
	Based on the discrete Gronwall inequality, the result is standard. Since $$\lim\limits_{n\to\infty}\big\|T^n\big\|^{\frac{1}{n}}=\rho<1,$$ 
    then for any sufficiently small $\varepsilon > 0$, there exists a constant $C_{\varepsilon}$ depending on $\varepsilon$ such that 
    $$\big\|T^n\big\|\le C_{\varepsilon}(\rho+\varepsilon/3)^n,\quad \text{for all}\; n\in\mathbb{N}.$$ 
    The condition $\lim\limits_{\|v\|_{X}\to0}\big\|Y(v)\big\|_{X}/\|v\|_{X}=0$ indicates that for any sufficiently small $\varepsilon$, there exists a small enough $\sigma_1$ such that for all $\|v\|_X\le\sigma_1$, there holds $$\big\|Y(v)\big\|_{X}\le\frac{\varepsilon}{3C_{\varepsilon}}\big\|v\big\|_{X}.$$ 
    Let $\sigma\le\frac{\sigma_1}{(1+C_{\varepsilon})}$, we use mathematical induction to prove $$\|v^n\|_{X}\le\sigma_1,\quad\text{for all}\; n\ge 0.$$  
    Obviously, $n=0$ is true. Then let us assume $\|v^{k}\|_{X}\le\sigma_1$ for all $k\le n-1\ (n\ge2)$. Hence, the following inequality holds for $k=n$
	\begin{align*}
		\big\|v^{n}\big\|_{X}&=\big\|Tv^{n-1}+Y(v^{n-1})\big\|_{X}\\
		&=\big\|T^2v^{n-2}+TY(v^{n-2})+Y(v^{n-1})\big\|_{X}
		=\Bigg\|T^{n}v^{0}+\sum\limits_{k=0}^{n-1}T^{n-1-k}Y(v^{k})\Bigg\|_{X}\\
		&\le\big\|T^{n}v^{0}\big\|_{X}+\sum\limits_{k=0}^{n-1}\big\|T^{n-1-k}\big\|\big\|Y(v^{k})\big\|_{X}\\
		&\le C_{\varepsilon}\big\|v^{0}\big\|_{X}(\rho+\varepsilon/3)^{n}+\sum\limits_{k=0}^{n-1}(\rho+\varepsilon/3)^{n-1-k}\frac{\varepsilon}{3}\big\|v^{k}\big\|_{X}\\
		\Longrightarrow&\qquad (\rho+\varepsilon/3)^{-n}\big\|v^{n}\big\|_{X}\le C_{\varepsilon}\big\|v^{0}\big\|_{X}+\sum\limits_{k=0}^{n-1}\frac{\varepsilon}{3\rho+\varepsilon}(\rho+\varepsilon/3)^{-k}\big\|v^{k}\big\|_{X}.
	\end{align*}
	Applying the classical discrete Gronwall inequality, we derive
	\begin{align*}
		(\rho+\varepsilon/3)^{-n}\big\|v^{n}\big\|_{X}\le C_{\varepsilon}\|v^{0}\|_{X}\Bigg(1+\frac{\varepsilon}{3\rho+\varepsilon}\Bigg)^{n}\Longrightarrow\big\|v^n\big\|_{X}\le C_{\varepsilon}\|v^{0}\|_{X}(\rho+\varepsilon)^{n}\le\sigma_1.
	\end{align*}
	This not only completes the induction but also proves the conclusion.
\end{proof}

With this, we are ready to prove the theorems.

\subsection{Proof of Theorem \ref{Exp-Lyapunov-Stab}}
\begin{proof}[Proof of {\bf Theorem \ref{Exp-Lyapunov-Stab}}]
		Define 
	\[
	e_n := |v_{n+1} - v_n| + |v_{n+1}|.
	\]
	Recall that in NAG the extrapolation point is given by
	\[
	y^n = x^n + \beta v_n, \quad x^{n+1} = y^n - \alpha \nabla f(y^n),
	\]
	so that
	\[
	x^{n+1} - y^n = v_{n+1}-\beta v_n, \qquad x^n - y^n = -\beta v_n.
	\]
	It follows that both displacements are controlled by $e_n$:
	\begin{align*}
		\|x^{n+1} - y^n\| &= \|v_{n+1} - \beta v_n\| 
		= \|(1-\beta)v_{n+1} + \beta(v_{n+1} - v_n)\| \le C \|e_n\|, \\
		\|x^n - y^n\| &= \|\beta v_n\| = \beta \|v_n-v_{n+1}+v_{n+1}\| \le C \|e_n\|.
	\end{align*}
	Under the $C^2$ regularity of $f$, we expand both $f(x^{n+1})$ and $f(x^n)$ around the intermediate point $y^n$ near $x_*$. Using the continuity of $\nabla^2 f$ at $x_*$, 
	we obtain
	\begin{align*}
		f(x^{n+1})
		&= f(y^n) + \langle \nabla f(y^n),\, x^{n+1} - y^n \rangle
		+ \frac{1}{2} \nabla^2 f(x_*)[x^{n+1} - y^n,\, x^{n+1} - y^n] + o(\|e_n\|^2), \\
		f(x^n) 
		&= f(y^n) + \langle \nabla f(y^n),\, x^n - y^n \rangle
		+ \frac{1}{2} \nabla^2 f(x_*)[x^n - y^n,\, x^n - y^n] + o(\|e_n\|^2).
	\end{align*}
	Subtracting these two expansions and using the identities
	\begin{align*}
		\big\langle\nabla f(y^n),\, x^{n+1} - x^n \big\rangle
		&= -\frac{1}{\alpha}\big\langle v_{n+1}-\beta v_n ,\, v_{n+1} \big\rangle\\
		&= -\frac{1}{\alpha}\|v_{n+1}\|^2 + \frac{\beta}{\alpha}\langle v_n, v_{n+1}\rangle  \\
		&= -\frac{1-\beta}{\alpha}\|v_{n+1}\|^2 -\frac{\beta}{2\alpha}\|v_{n+1}\|^2 + \frac{\beta}{2\alpha}\|v_n\|^2 
		- \frac{\beta}{2\alpha}\|v_{n+1} - v_n\|^2
	\end{align*}
	and
	\begin{align*}
		&\frac{1}{2}\nabla^2f(x_*)[x^{n+1} - y^n,x^{n+1} - y^n] - \frac{1}{2}\nabla^2f(x_*)[x^n - y^n,x^n - y^n] \\
		=& \frac{1}{2}\nabla^2f(x_*)[v_{n+1},v_{n+1}-2\beta v_n]\\
		=& \frac{1 - \beta}{2}\nabla^2f(x_*)[v_{n+1},v_{n+1}] 
		- \frac{\beta}{2}\nabla^2f(x_*)[v_n,v_n]
		+ \frac{\beta}{2}\nabla^2f(x_*)[v_{n+1} - v_n,v_{n+1} - v_n]\\
		=& \frac{\beta^2}{2(1+\beta)}\nabla^2f(x_*)[v_{n+1},v_{n+1}]-\frac{\beta^2}{2(1+\beta)}\nabla^2f(x_*)[v_n,v_n]+ \frac{1 + \beta - 2\beta^2}{2(1+\beta)} \nabla^2f(x_*)[v_{n+1},v_{n+1}] 
		\\
		&\hspace{1.5cm}- \frac{\beta}{4(1+\beta)} \nabla^2f(x_*)[v_{n+1} + v_n,v_{n+1} + v_n] + \frac{2\beta^2 + \beta}{4(1+\beta)} \nabla^2f(x_*)[v_{n+1} - v_n,v_{n+1} - v_n],
	\end{align*}
	we arrive at
	\begin{align*}
		\mathcal{L}(x^{n+1}, v_{n+1})
		&= \mathcal{L}(x^n, v_n) 
		- \frac{1 - \beta}{\alpha} \|v_{n+1}\|^2 
		- \frac{\beta}{2\alpha} \|v_{n+1} - v_n\|^2 \\
		&\quad - \frac{\beta}{4(1+\beta)} \nabla^2f(x_*)[v_{n+1} + v_n,v_{n+1} + v_n]
		+ \frac{1 + \beta - 2\beta^2}{2(1+\beta)} \nabla^2f(x_*)[v_{n+1},v_{n+1}] \\
		&\quad + \frac{2\beta^2 + \beta}{4(1+\beta)} \nabla^2f(x_*)[v_{n+1} - v_n,v_{n+1} - v_n] + o(\|e_n\|^2).
	\end{align*}
	Since $\nabla^2 f(x_*) \preceq L I$, we have $\nabla^2 f(x_*)[v,v] \le L \|v\|^2$ for all $v \in \mathbb{R}^d$. Consequently, we further get 
	\begin{align*}
		\mathcal{L}(x^{n+1}, v_{n+1})
		&\le \mathcal{L}(x^n, v_n) 
		- \left( \frac{1 - \beta}{\alpha} - \frac{1 + \beta - 2\beta^2}{2(1+\beta)} L \right) \|v_{n+1}\|^2 \\
		&\quad - \left( \frac{\beta}{2\alpha} - \frac{2\beta^2 + \beta}{4(1+\beta)} L \right) \|v_{n+1} - v_n\|^2 + o(\|e_n\|^2).
	\end{align*}
	Assume that $\{x^n\}_{n\in\mathbb{N}} \subset \mathcal{B}_{\sigma_1}(x_*)$ for some sufficiently small $\sigma_1 > 0$. By the triangle inequality, we have
	\[
	\|v_n\| = \|x^n - x^{n-1}\| \le \|x^n - x_*\| + \|x^{n-1} - x_*\| \le 2\sigma_1.
	\]
	For sufficiently small $\sigma_1 > 0$, the $C^2$ regularity of $f$ ensures that all higher-order remainder terms are dominated by $\varepsilon \|e_n\|^2$ for a sufficiently small $\varepsilon > 0$. Consequently, the Lyapunov function satisfies
	\begin{align*}
		\mathcal{L}(x^{n+1}, v_{n+1})
		&\le \mathcal{L}(x^n, v_n) 
		- \left( \frac{1 - \beta}{\alpha} - \frac{1 + \beta - 2\beta^2}{2(1+\beta)} (L + \varepsilon) \right) \|v_{n+1}\|^2 \\
		&\quad - \left( \frac{\beta}{2\alpha} - \frac{2\beta^2 + \beta}{4(1+\beta)} (L + \varepsilon) \right) \|v_{n+1} - v_n\|^2.
	\end{align*}
	The coefficients of $\|v_{n+1}\|^2$ and $\|v_{n+1} - v_n\|^2$ are positive provided that
	\[
	\frac{1 - \beta}{\alpha} > \frac{1 + \beta - 2\beta^2}{2(1+\beta)} (L + \varepsilon), \quad
	\frac{\beta}{2\alpha} > \frac{2\beta^2 + \beta}{4(1+\beta)} (L + \varepsilon),
	\]
	which are equivalent to
	\[
	\alpha \in \left( 0, \frac{2(\beta+1)}{(L+\varepsilon)(2\beta+1)} \right), \quad \text{for all } \beta \in [0,1).
	\]
	Under this condition, there exists a constant $C_{\varepsilon,\alpha,\beta} > 0$ such that
	\[
	\mathcal{L}(x^{n+1}, v_{n+1}) 
	\le \mathcal{L}(x^n, v_n) - C_{\varepsilon,\alpha,\beta} \big( \|v_{n+1}\|^2 + \|v_{n+1} - v_n\|^2 \big).
	\]
	Moreover, it can be verified that
	\[
	\frac{2(\beta+1)}{(L+\varepsilon)(2\beta+1)} \le \frac{\beta+1}{L\beta},
	\]
	which shows that the chosen $\alpha$ is compatible with the condition in \textbf{Lemma~\ref{Lyapunov-Inf-Sup}} when $\sigma_1$ is sufficiently small.
	Next, we invoke the higher-order expansion of NAG:
	\[
	v_{n+1} = \beta v_n - \alpha \nabla f(x^n) - \alpha \beta \nabla^2 f(x^n)[v_n] + o(\|v_n\|),
	\]
	which implies that, near $x_*$,
	\[
	\mathcal{L}(x^n, v_n) \approx \|\nabla f(x^n)\|^2 + \|v_n\|^2 
	\le C \big( \|v_{n+1}\|^2 + \|v_{n+1} - v_n\|^2 \big).
	\]
	Consequently, there exists a constant $C'_{\varepsilon,\alpha,\beta}>0$, depending on $\varepsilon$, $\alpha$, and $\beta$, such that
	\[
	\mathcal{L}(x^{n+1}, v_{n+1}) 
	\le \mathcal{L}(x^n, v_n) - C'_{\varepsilon,\alpha,\beta} \, \mathcal{L}(x^n, v_n).
	\]
	Since $\mathcal{L}(x,v) \ge 0$ and vanishes only on $\mathcal{S}\times\{0\}$, the above inequality implies $C'_{\varepsilon,\alpha,\beta}\in(0,1)$. Defining the contraction factor as
	\[
	\rho_{\varepsilon,\alpha,\beta} := 1 - C'_{\varepsilon,\alpha,\beta} \in (0,1),
	\]
	we obtain
	\[
	0 \le \mathcal{L}(x^{n+1}, v_{n+1}) \le \rho_{\varepsilon,\alpha,\beta} \, \mathcal{L}(x^n, v_n).
	\]
	
	It remains to verify that the entire sequence $\{x^n\}_{n\in\mathbb{N}}$ remains in $\mathcal{B}_{\sigma_1}(x_*)$, provided $x^0, x^1$ are sufficiently close to $x_*$.
	To this end, choose $\sigma \in (0, \sigma_1)$ such that $x^0, x^1 \in \mathcal{B}_{\sigma}(x_*)$.	Then the claim holds for $n = 0, 1$.
	
	We now proceed by induction. Assume that $x^k \in \mathcal{B}_{\sigma_1}(x_*)$ for all $k \le n$ with $n \ge 1$.  
	Then all local estimates derived above—including the Lyapunov decay estimate—are valid up to step $n$, yielding
	\[
	\mathcal{L}(x^{n+1}, v_{n+1}) \le \rho \, \mathcal{L}(x^n, v_n) \le \cdots \le \rho^{n} \mathcal{L}(x^1, v_1),
	\]
	for some $\rho \in (0,1)$.
	Using this geometric decay and the local Lipschitz continuity of the gradient, we bound the deviation of $x^{n+1}$ from $x_*$ as follows:
	\begin{align*}
		\|x^{n+1} - x_*\|
		&\le \|x^0 - x_*\| + \sum_{j=0}^{n} \|x^{j+1} - x^j\|\\
		&= \|x^0 - x_*\| + \sum_{j=0}^{n} \|v_{j+1}\|\\
		&\le \|x^0 - x_*\| + C \sum_{j=0}^{n} \sqrt{\mathcal{L}(x^{j+1}, v_{j+1})} \\
		&\le \sigma + C \sqrt{\mathcal{L}(x^1, v_1)} \sum_{j=0}^{n} \rho^{j/2} \\
		&\le \sigma + C \sqrt{\mathcal{L}(x^1, v_1)} \cdot \frac{1}{1 - \sqrt{\rho}} \nonumber\\
		&\le \sigma + C \big(\|\nabla f(x^1)\|+\|x^1-x_*\|+\|x^0-x_*\|\big)\\
		&\le \sigma + C \sigma \\
		&\le C \sigma,
	\end{align*}
	where $C > 1$ is a uniform constant independent of $n$.
	
	Choosing $\sigma < \sigma_1 / C$, we obtain $x^{n+1} \in \mathcal{B}_{\sigma_1}(x_*)$.  This completes the induction step and confirms that the entire sequence $\{x^n\}_{n\in\mathbb{N}} \subseteq \mathcal{B}_{\sigma_1}(x_*)$, thereby validating all local estimates used above.
\end{proof}
\subsection{Proof of Theorem \ref{Opt-convergence}}
\begin{proof}[Proof of {\bf Theorem \ref{Opt-convergence}}]  By \textbf{Lemma~\ref{Lyapunov-Inf-Sup}} and \textbf{Theorem~\ref{Exp-Lyapunov-Stab}}, 
	for every sufficiently small $\varepsilon > 0$, there exists $\sigma \in (0, \sigma_1)$ such that, 
	for any $x^0, x^1 \in \mathcal{B}_{\sigma}(x_*)$, the NAG iterates satisfy $\{x^n\}_{n\in\mathbb{N}} \subset \mathcal{B}_{\sigma_1}(x_*)$ and hence the difference sequence $\{v_n\}_{n\in\mathbb{N}}$ converge linearly to $0$, 
	provided $(\alpha, \beta)$ satisfies
	\[
	\alpha \in \left(0, \frac{2(1 + \beta)}{(2\beta + 1)(L + \varepsilon)}\right), \quad \beta \in (0,1).
	\]
	Specifically, there exists a constant $C_{\varepsilon} > 0$ such that  
	\[
	\|x^{n+1} - x^n\| \le C \sqrt{\mathcal{L}(x^1,v_1)} \rho^{n},
	\]
	where $\rho \in (0,1)$ depends on $\varepsilon,\alpha$, and $\beta$.
	
	This geometric decay of successive differences implies that the iterates $\{x^n\}_{n \in \mathbb{N}}$ themselves converge linearly to the local minimizer $x_*'$. Indeed, for any integers $m > n \geq 1$, we have
	\begin{align*}
		\|x^m - x^n\| 
		&\leq \sum_{k=n+1}^{m} \|x^k - x^{k-1}\| 
		\leq \sum_{k=n+1}^{\infty} \|x^k - x^{k-1}\| \\
		&\leq C \sqrt{\mathcal{L}(x^1, v_1)} \sum_{k=n+1}^{\infty} \rho^{\,k}
		= C \sqrt{\mathcal{L}(x^1, v_1)} \cdot \frac{\rho^{\,n+1}}{1 - \rho}.
	\end{align*}
	Hence $\{x^n\}_{n \in \mathbb{N}}$ is a Cauchy sequence in $\mathbb{R}^d$, and therefore converges to some limit point $x_*' \in \mathcal{B}_{\sigma_1}(x_*)$. Moreover, passing to the limit in the above estimate yields the linear convergence rate
	\[
	\|x^n-x_*'\|\le C \sqrt{\mathcal{L}(x^1, v_1)}\rho^n\le C \big(\|x^1-x_*'\|+\|x^0-x_*'\|\big)\rho^n.
	\]
	
	The optimal local linear convergence rate is established below. A local linearization of the NAG iteration at $x_*'$ yields
	\begin{align*}
		x^{n+1} - x_*'
		&= x^n - x_*' +\beta(x^{n}-x^{n-1})- \alpha\nabla^2f(x_*')[x^n - x_*'+\beta(x^{n}-x^{n-1})]\\
		& + o(\|x^n - x_*'\|+\|x^{n}-x^{n-1}\|) \\
		&= (I - P_{x_*'})(x^n - x_*'+\beta(x^n-x^{n-1})) + (1+\beta)\mathcal{G}_\alpha(x_*')P_{x_*'}(x^n - x_*')\\ &-\beta\mathcal{G}_\alpha(x_*')P_{x_*'}(x^{n-1} - x_*')+ o(\|x^n - x_*'\|+\|x^n-x^{n-1}\|).
	\end{align*}
	From this expansion, the following decoupled asymptotic relations follow:
	\begin{align*}
		(I - P_{x_*'})(x^{n+1} - x^n) &=\beta(I - P_{x_*'})(x^{n} - x^{n-1})+ o(\|x^n - x_*'\|+\|x^n-x^{n-1}\|),\\
		P_{x_*'}(x^{n+1} - x_*') &= (1+\beta)\mathcal{G}_\alpha(x_*')P_{x_*'}(x^n - x_*')-\beta\mathcal{G}_\alpha(x_*')P_{x_*'}(x^{n-1} - x_*')\\
		&\hspace{5cm}+ o(\|x^n - x_*'\|+\|x^n-x^{n-1}\|).	
	\end{align*}
	Telescopic summation of the first relation gives
	\begin{align}\label{telescoping-identity}
		(I - P_{x_*'})(x^n - x_*')-\beta(I - P_{x_*'})(x^{n-1} - x_*')
		&= \sum_{k=n}^\infty (I - P_{x_*'})(x^k - x^{k+1}-\beta(x^{k-1}-x^k))
	\end{align}
	with $\big\|(I - P_{x_*'})(x^k - x_*^{k+1})-\beta(I - P_{x_*'})(x^{k-1} - x_*^k)\big\|
	\to 0$ as $(\| x^k - x_*' \|+\|x^k-x^{k-1}\|)\to0$. In particular, for every $\sigma>0$ there exists $\delta_\sigma\to0^+$ as
	$\sigma\to0$ such that
	\begin{align*}
		\big\| (I - P_{x_*'})(x^k - x^{k+1}-\beta(x^{k-1}-x^k)) \big\| 
		\le \delta_\sigma (\| x^k - x_*' \|+\|x^k-x^{k-1}\|)
	\end{align*}
	whenever 
	\[
	\| x^k - x_*' \|+\| x^k - x^{k-1} \|\le \sigma.
	\]
	Hence, from the telescoping identity \eqref{telescoping-identity} and for $n$ large enough so that 
	\[
	\| x^k - x_*' \|+\|x^k-x^{k-1}\|\le\sigma,\quad\text{ for all}\; k\ge n,
	\] 
	it follows that
	\begin{align}
		\label{estimate-telescoping-identity}
		\bigl\|(I - P_{x_*'})(x^n - x_*')-\beta(I - P_{x_*'})(x^{n-1} - x_*')\bigr\|
		\le \delta_\sigma \sum_{k=n}^\infty \bigl(\| x^k - x_*' \|+\| x^k - x^{k-1} \|\bigr).
	\end{align}
	Given the linear convergence of $\{x^n\}_{n\in\mathbb{N}}$ and $\{x^n-x^{n-1}\}_{n\in\mathbb{N}}$, there exists $\rho \in (0,1)$ such that
	\[
	\|x^k - x_*'\|+\|x^k - x^{k-1}\| \le C \rho^{k-n}\left(\|x^n - x_*'\|+\|x^n - x^{n-1}\|\right)\quad\text{ for all}\;\; k\ge n,
	\]
	Consequently,
	\begin{align*}
		{
			\bigl\|(I - P_{x_*'})(x^n - x_*')-\beta(I - P_{x_*'})(x^{n-1} - x_*')\bigr\|
			\overset{\eqref{estimate-telescoping-identity}}{\le}  }
		\delta_{\sigma} \sum_{k=n}^\infty \bigl(\|x^k - x_*'\|+\|x^k - x^{k-1}\|\bigr)&\\
		\le C \delta_{\sigma} \bigl(\|x^n - x_*'\|+\|x^n - x^{n-1}\|\bigr)&,
	\end{align*}
	where $ \delta_{\sigma}\to 0^+$ as $\sigma\to0^+$. It follows that
	\begin{align*}
		&(I - P_{x_*'})(x^n - x_*')-\beta(I - P_{x_*'})(x^{n-1} - x_*')
		= o(\|x^n - x_*'\|+\|x^n-x^{n-1}\|),\\
		\Longrightarrow\;&(I - P_{x_*'})(x^n - x_*')=\frac{\beta}{1-\beta}(I - P_{x_*'})(x^{n-1} - x^n)+o(\|x^n - x_*'\|+\|x^n-x^{n-1}\|).
	\end{align*}
	Using the identities
	\begin{align*}
		x^n-x_*'&=P_{x_*'}(x^n - x_*')+(I - P_{x_*'})(x^n - x_*'),\\
		x^n-x^{n-1}&=P_{x_*'}(x^n - x_*')- P_{x_*'}(x^{n-1} - x_*')+(I - P_{x_*'})(x^n - x^{n-1}),
	\end{align*}
	and applying the triangle inequality, we obtain
	\begin{align*}
		&o\; (\|x^n - x_*'\|+\|x^n-x^{n-1}\|)\\
		=&o\left(\big\|P_{x_*'}(x^n - x_*')\big\|+\big\|P_{x_*'}(x^{n-1} - x_*')\big\|+\big\|(I-P_{x_*'})(x^n - x^{n-1})\big\|\right).	
	\end{align*}
	Define the augmented state vector
 \[
     X^{n}:=\begin{pmatrix}
			P_{x_*'}(x^{n} - x_*') \\
			P_{x_*'}(x^{n-1} - x_*') \\
			(I - P_{x_*'})(x^{n} - x^{n-1})
		\end{pmatrix}.
 \]
 We can then rewrite the iteration using the matrix $G_{\alpha,\beta}(x_*')$ from  \eqref{G-block-oper} as
	\begin{align*}
		X^{n+1}
		=
		G_{\alpha,\beta}(x_*')
		X^{n}
		+
		o\big(\|X^n\|\big),
	\end{align*}
 where 
 \[
 \|X^n\|=\sqrt{\|P_{x_*'}(x^{n} - x_*')\|^2+\|P_{x_*'}(x^{n-1} - x_*')\|^2+\|(I - P_{x_*'})(x^{n} - x^{n-1})\|^2}.
 \]
	By \textbf{Lemma~\ref{G-spectrum}} and \textbf{Lemma~\ref{Nonlinear-Iteration}}, we immediately obtain the following local linear convergence rate:
	\begin{align*}
		\left\|
		X^{n+1}
		\right\|
		\le
		C_{\varepsilon}\left\|
		X^{n}
		\right\|
		\left( \rho\bigl(G_{\alpha,\beta}(x_*')|_{\mathcal{N}_{x_*'}}\bigr) + \varepsilon \right)^n,
		\quad \forall\, n \ge 1,
	\end{align*}
	Together with the inequality
	\begin{align*}
		\|x^n-x^{n-1}\|&\le \big\|P_{x_*'}(x^n - x_*')\big\|+\big\| P_{x_*'}(x^{n-1} - x_*')\big\|+\big\|(I - P_{x_*'})(x^n - x^{n-1})\big\|\\
  &=\|X^n\|,
	\end{align*} 
	and the geometric decay estimate of $\{X^n\}_{n\in\mathbb{N}}$, we deduce that for all parameters  $ (\alpha,\beta) $  satisfying
	\begin{align*}
		\alpha\in\left(0,\frac{2(\beta+1)}{(L+\varepsilon)(2\beta+1)}\right),\; \beta\in[0,1),
	\end{align*}
	the following locally linear convergence rate holds
	\[
	\| x^n - x^{n-1} \|
	\le C_{\varepsilon}\big(\| x^{1} - x_*' \|+\| x^{0} - x_*' \|\big)\left( \rho\bigl(G_{\alpha,\beta}(x_*')|_{\mathcal{N}_{x_*'}}\bigr) + \varepsilon \right)^n,
	\quad \forall\; n \ge 1.
	\]
   This suggests that
   \begin{align*}
    \| x^n - x_*' \|&=\sum_{k=n+1}^{\infty}\|x^k-x^{k-1}\|\\
    &\le C_{\varepsilon}\big(\| x^{1} - x_*' \|+\| x^{0} - x_*' \|\big)\left( \rho\bigl(G_{\alpha,\beta}(x_*')|_{\mathcal{N}_{x_*'}}\bigr) + \varepsilon \right)^n,
	\quad \forall\; n \ge 1.
   \end{align*}
	In particular, the optimal contraction factor is attained when 
	\begin{align*}
		\alpha_{\mathrm{opt}} = \frac{4}{3L + \mu}, \qquad 
		\beta_{\mathrm{opt}} = \frac{\sqrt{3L + \mu} - 2\sqrt{\mu}}{\sqrt{3L + \mu} + 2\sqrt{\mu}},
	\end{align*}
	and the minimal value is given by
	\begin{align*}
		\rho_{\mathrm{opt}}\bigl( G_{\alpha,\beta}(x_*')|_{\mathcal{N}_{x_*'}} \bigr) 
		= \frac{\sqrt{3L + \mu} - 2\sqrt{\mu}}{\sqrt{3L + \mu}}.
	\end{align*}
\end{proof}
\subsection{Proof of Theorem \ref{Opt-convergence-flow}}
\begin{proof}[Proof of {\bf Theorem \ref{Opt-convergence-flow}}]
	The proof follows the same strategy as that of Theorem~\ref{Opt-convergence}: we first establish local exponential convergence of the trajectories via a Lyapunov function argument. Specifically, for any sufficiently small $\varepsilon > 0$, we employ the Lyapunov function constructed in \cite[Theorem~4]{2017Polyak}, namely
	\[
	\widetilde{\mathcal{L}}(x,v) := f(x) - f(x_*) + \frac{\gamma}{\gamma^2 + 2(L+\varepsilon)} \langle \nabla f(x), v \rangle + \frac{L+\varepsilon}{\gamma^2 + 2(L+\varepsilon)} \|v\|^2.
	\]
	Since $f \in C^2$, for any given $\varepsilon > 0$, there exists a sufficiently small open neighborhood $U_1$ of $x_*$ such that
	\[
	\nabla^2 f(x)[v,v] \le (L + \varepsilon) \|v\|^2, \quad \text{for all } x \in U_1 \text{ and all } v \in \mathbb{R}^d.
	\]

	Assuming that the trajectory remains in $U_1$ for all $t \ge 0$, the argument in \cite[Theorem~4]{2017Polyak} directly yields the exponential decay estimate
	\begin{align*}
		\big\|\nabla f(x(t))\big\|^2+\|\dot{x}(t)\|^2\approx\widetilde{\mathcal{L}}(x(t), \dot{x}(t)) \le e^{-\rho t} \, \widetilde{\mathcal{L}}(x(0), \dot{x}(0)), \quad \forall\, t \ge 0,
	\end{align*}
	for some constant $\rho > 0$ depending on $\varepsilon,\gamma$.
	Finally, by choosing the initial state $(x(0), \dot{x}(0))$ in a sufficiently small open neighborhood $\mathcal{U}\subset\mathbb{R}^d \times \mathbb{R}^d$ of $(x_*, 0)$, one can guarantee that $x(t) \in U_1$ for all $t \geq 0$. This step follows verbatim from the proof of {\bf Theorem~\ref{Exp-Lyapunov-Stab}} and is therefore omitted here. This implies that for all $t' \ge t \ge 0$,
	\begin{align}\label{dotx-x}
		\nonumber\|x(t') - x(t)\| 
		&= \left\| \int_t^{t'} \dot{x}(s) \, ds \right\| 
		\le \int_t^{t'} \|\dot{x}(s)\| \, ds \\
		&\le \int_t^{t'} C e^{-\rho s} \sqrt{\widetilde{\mathcal{L}}(x(0), \dot{x}(0))} \, ds 
		\le C e^{-\rho t} \sqrt{\widetilde{\mathcal{L}}(x(0), \dot{x}(0))},
	\end{align}
	where we used the exponential decay of $\|\dot{x}(t)\|$. By the Cauchy criterion, $\{x(t)\}_{t \ge 0}$ converges to some limit point $x_*' \in U_1$ as $t \to \infty$, and the convergence is locally exponential, i.e.,
	\begin{align*}
		\|x(t)-x_*'\|\le Ce^{-\rho t}\sqrt{\widetilde{\mathcal{L}}(x(0), \dot{x}(0))}\le Ce^{-\rho t}\big(\|x(0)-x_*'\|+\|\dot{x}(0)\|\big)
	\end{align*}
	
	Next, we linearize the dynamics~\eqref{Accelerated-flow} around the limit point $x_*'$. Since $f \in C^2$, we have
	\[
	\ddot{x}(t) + \gamma \dot{x}(t) = -\nabla^2 f(x_*') (x(t) - x_*') + o(\|x(t) - x_*'\|), \quad \text{as } t \to \infty.
	\]
	Decomposing the above equation yields the asymptotic relations:
	\begin{align*}
		(I - P_{x_*'}) \ddot{x}(t) &= -\gamma (I - P_{x_*'}) \dot{x}(t) + o(\|x(t) - x_*'\|), \\
		P_{x_*'} \ddot{x}(t) + \gamma P_{x_*'} \dot{x}(t) &= -\nabla^2 f(x_*') P_{x_*'} (x(t) - x_*') + o(\|x(t) - x_*'\|).
	\end{align*}
	Denote the remainder term by $R(t) := o(\|x(t) - x_*'\|)$ as $t \to \infty$. 
	Integrating the first equation of \eqref{Accelerated-flow} from $t$ to $t'$ yields
	\begin{align*}
		(I - P_{x_*'})(\dot{x}(t') - \dot{x}(t)) 
		= -\gamma (I - P_{x_*'})(x(t') - x(t)) + \int_t^{t'} R(s) \, ds.
	\end{align*}
	Taking the limit as $t' \to \infty$, and using the facts that
	\[
	\lim_{t' \to \infty} x(t') = x_*', \qquad \lim_{t' \to \infty} \dot{x}(t') = 0,
	\]
	we obtain
	\begin{align*}
		(I - P_{x_*'}) \dot{x}(t) 
		= -\gamma (I - P_{x_*'})(x(t) - x_*') - \int_t^{\infty} R(s) \, ds.
	\end{align*}
	Moreover, since $\|R(s)\| \le \sigma_s \|x(s) - x_*'\|$ with $\sigma_s \to 0$ as $s \to \infty$, and 
	\[
	\|x(s)-x_*'\|\le  Ce^{-\rho (s-t)}\big(\|x(t)-x_*'\|+\|\dot{x}(t)\|\big)\quad\text{for all } s \ge t \text{ with } t \text{ sufficiently large},
	\]
	it follows that
	\begin{align*}
		\left\| \int_t^{\infty} R(s) \, ds \right\| 
		&\le \int_t^{\infty} \|R(s)\| \, ds 
		\le \int_t^{\infty} \sigma_s \|x(s) - x_*'\| \, ds \\
		&\le C \big(\|x(t)-x_*'\|+\|\dot{x}(t)\|\big)\left( \sup_{s \ge t} \sigma_s \right) \int_t^{\infty} e^{-\rho (s-t)} \, ds \\
		&= \frac{C}{\rho} \big(\|x(t)-x_*'\|+\|\dot{x}(t)\|\big) \sup_{s \ge t} \sigma_s \\
		&= o\big(\|x(t)-x_*'\|+\|\dot{x}(t)\|\big) \quad \text{as } t \to \infty.
	\end{align*}
	Hence,
	\begin{equation}\label{eq:projected-relation}
		(I - P_{x_*'})(x(t) - x_*')
		= -\frac{1}{\gamma} (I - P_{x_*'}) \dot{x}(t)  + o\big(\|x(t)-x_*'\|+\|\dot{x}(t)\|\big).
	\end{equation}
	Now decompose the state using the orthogonal projection $P_{x_*'}$. We have
	\begin{align*}
		x(t)-x_*' &= P_{x_*'}(x(t) - x_*') + (I - P_{x_*'})(x(t) - x_*'), \\
		\dot{x}(t) &= P_{x_*'}\dot{x}(t) + (I - P_{x_*'})\dot{x}(t).
	\end{align*}
	Combining these with \eqref{eq:projected-relation} and applying the triangle inequality, the asymptotic relation
	\[
	o\big(\|x(t) - x_*'\|+\|\dot{x}(t)\|\big)
	= o\left( \big\|P_{x_*'}(x(t) - x_*')\big\| + \big\|P_{x_*'}\dot{x}(t)\big\| + \big\|(I - P_{x_*'})\dot{x}(t)\big\| \right)
	\]
	holds as $t \to \infty$.
	
	Define the augmented state vector
	\[
	X(t) := 
	\begin{pmatrix}
		P_{x_*'}(x(t) - x_*') \\
		P_{x_*'}\dot{x}(t) \\
		(I - P_{x_*'})\dot{x}(t)
	\end{pmatrix}.
	\]
	Then, by decomposing the accelerated flow \eqref{Accelerated-flow} around the equilibrium $(x_*, 0)$, the dynamics of $X(t)$ are governed asymptotically by the linear system
	\[
	\frac{\text{d}}{\text{d}t} X(t) = AX(t) + o(\|X(t)\|)
	\]
	with 
	\[
	A:=\begin{pmatrix}
		\mathbf{0} & I & \mathbf{0} \\
		-\nabla^2 f(x_*') & -\gamma I & \mathbf{0} \\
		\mathbf{0} & \mathbf{0} & -\gamma I
	\end{pmatrix}
	\]
	Observe that the spectrum of the matrix $A$ restricted to the invariant subspace  
	\[
	\mathcal{N}_{x_*'}=N_{x_*'}\mathcal{S}_{U} \times N_{x_*'}\mathcal{S}_{U} \times T_{x_*'}\mathcal{S}_{U}\subset\mathbb{R}^d\times\mathbb{R}^d\times\mathbb{R}^d
	\]  
	is given by  
	\[
	\operatorname{Spec}(A|_{\mathcal{N}_{x_*'}}) = \left\{ \frac{-\gamma \pm \sqrt{\gamma^2 - 4\gamma\lambda_i}}{2} \;:\; \lambda_i \in [\mu, L] \right\} \cup \{-\gamma\}.
	\]  
	Consequently, the decay rate is determined by the spectral bound  
	\[
	-\rho_{\gamma} = -\frac{\gamma - \sqrt{\max\{0,\, \gamma^2 - 4\mu\}}}{2}.
	\]
	Standard results (see, e.g., \cite[Lemma 7]{1964P} or \cite[Lemma 4.5]{2026HBOpt}) imply that for any $\varepsilon > 0$, there exists a constant $C_{\varepsilon} > 0$ such that  
	\[
	\|X(t)\| \le C_{\varepsilon} e^{-(\rho_{\gamma} - \varepsilon)t} \|X(0)\|,\quad\forall \; t\ge0.
	\]
	Combining this exponential decay estimate with the inequality  
	\[
	\|\dot{x}(t)\| \le \big\|P_{x_*'}\dot{x}(t)\big\| + \big\|(I - P_{x_*'})\dot{x}(t)\big\|\le \|X(t)\|,
	\]  
	and repeating the argument used in \eqref{dotx-x}, we obtain the following quantitative exponential convergence rate:  
	\[
	\|x(t) - x_*'\| \le C_{\varepsilon} e^{-(\rho_{\gamma}-\varepsilon) t} \big( \|x(0) - x_*'\| + \|\dot{x}(0)\| \big),\quad\forall \; t\ge0.
	\]
	
	In particular, when the damping parameter is chosen as $\gamma = 2\sqrt{\mu}$, the optimal exponential decay rate is achieved:  
	\[
	\rho_{\mathrm{opt}}  = \sqrt{\mu}.
	\]
\end{proof}
\section{Numerical experiments}\label{Sec4}
In this section, we will introduce a series of numerical experiments to verify our theoretical results.
\subsection{Test function and its theoretical properties}
To numerically validate the optimal local linear convergence of NAG established in {\bf Theorem~\ref{Opt-convergence}}, we consider a class of nonconvex objective functions that are smooth, possess a manifold of minimizers, and satisfy the local PL condition near the minimizer set by the local Morse--Bott condition.

Specifically, let $u \in \mathbb{R}^k$, $v \in \mathbb{R}^m$, and define  
\[
f(u, v) := \frac{1}{2} u^\top \operatorname{diag}(\lambda_1, \lambda_2, \dots, \lambda_k) u + \frac{1}{8} \lambda_{k+1} \big( \|v\|^2 - 1 \big)^2,
\]
where the parameters satisfy  
\[
0<\lambda_1\le\lambda_2\le\cdots\lambda_{i}\le\cdots\lambda_k\le\lambda_{k+1} \quad \text{for all } i = 1, \dots, k+1.
\]

We now verify the key properties of $f$.

\medskip

\noindent\textbf{(i) Nonconvexity.}

The Hessian of $f$ is given by
\[
\nabla^2 f(u,v) =
\begin{pmatrix}
	\operatorname{diag}(\lambda_1,\dots,\lambda_k) & 0 \\
	0 & \lambda_{n+1}\left( \tfrac{1}{2}(\|v\|^2 - 1) I_m + v v^\top \right)
\end{pmatrix}.
\]
At the point $(u,v) = (0,0)$, the lower-right block reduces to $-\tfrac{\lambda_{k+1}}{2} I_m$, which is negative definite. Hence $\nabla^2 f(0,0)$ has negative eigenvalues, and $f$ is nonconvex.

\medskip

\noindent\textbf{(ii) Local non-strong convexity and Morse--Bott structure.}

The global minimizer set of $f$ is the smooth submanifold
\[
\mathcal{M} := \left\{ (u, v) \in \mathbb{R}^k \times \mathbb{R}^m : u = 0,\; \|v\|^2 = 1 \right\}.
\]
The gradient vanishes precisely on $\mathcal{M}$ and at the isolated saddle point $(0,0)$. Thus, $f$ has no other local minima, and it is not strongly convex in any neighborhood of a minimizer since the minimizers are not isolated.

At any $(0, v_*) \in \mathcal{M}$ (so $\|v_*\| = 1$), the Hessian simplifies to
\[
\nabla^2 f(0, v_*) =
\begin{pmatrix}
	\operatorname{diag}(\lambda_1,\dots,\lambda_k) & 0 \\
	0 & \lambda_{k+1} v_* v_*^\top
\end{pmatrix}.
\]
Its spectrum consists of:
\begin{itemize}
	\item positive eigenvalues $\lambda_1, \dots, \lambda_k$ from the $u$-block;
	\item one positive eigenvalue $\lambda_{k+1}$ in the normal direction to the sphere $\|v\|=1$;
	\item zero eigenvalues (multiplicity $m-1$) along the tangent directions of the sphere.
\end{itemize}
Hence, $\nabla^2 f(0, v_*)$ is positive semidefinite with constant rank, and its kernel coincides with the tangent space of $\mathcal{M}$. This confirms that $f$ is a Morse--Bott function. Moreover, since the local PL condition is equivalent to the local Morse--Bott condition for smooth functions, $f$ naturally satisfies the local PL condition near $\mathcal{M}$.

Finally, the smallest positive eigenvalue and the largest eigenvalue of $\nabla^2 f$ on $\mathcal{M}$ are
\begin{align}\label{Test-L-mu}
	\mu := \min_{i=1,2,\cdots,k+1}\{\lambda_i\}=\lambda_1, \qquad 
	L := \max_{i=1,2,\cdots,k+1}\{\lambda_i\}=\lambda_{k+1}.
\end{align}

To sum up, $f$ satisfies all structural assumptions required by our theory.

\subsection{Numerical results}
We now instantiate the test function $f$ with the following concrete parameters:
\[
k = 10^{2}, 10^{3}, 10^{4}, 10^{5}, \quad m = 10, \quad \lambda_i = i \quad \text{for } i = 1, 2, \dots, k+1.
\]
Consequently, the curvature constants on the minimizer manifold $\mathcal{M}$ are
\[
\mu = \lambda_1 = 1, \qquad L = \lambda_{k+1} = k+1,
\]
yielding a condition number $\kappa = L / \mu = 1/(k+1)$.

The NAG iterates are initialized at a point $(u^0, v^0)$ in a neighborhood of $\mathcal{M}$:
\[
u^0 = \varepsilon_u \cdot \xi_u, \quad v^0 = (1 + \varepsilon_v) \cdot \frac{\xi_v}{\|\xi_v\|},
\]
where $\xi_u \in \mathbb{R}^{k}$ and $\xi_v \in \mathbb{R}^{m}$ are standard Gaussian random vectors, and we set $\varepsilon_u = \varepsilon_v = 10^{-4}$ to ensure the initial point lies sufficiently close to $\mathcal{M}$ for local convergence to manifest.

We apply NAG with the optimal parameters, namely
\begin{align*}
	\alpha_{\mathrm{opt}} = \frac{4}{3L + \mu}, \qquad 
	\beta_{\mathrm{opt}} = \frac{\sqrt{3L + \mu} - 2\sqrt{\mu}}{\sqrt{3L + \mu} + 2\sqrt{\mu}},
\end{align*}
which yield the asymptotically optimal accelerated convergence rate
\begin{align*}
	\rho_{\mathrm{opt}} = \frac{\sqrt{3L + \mu} - 2\sqrt{\mu}}{\sqrt{3L + \mu}}.
\end{align*}
 These choices are consistent with the local analysis in our theorem.
The algorithm is terminated when 
\[
 \big|f(u^n, v^n)\big| < 10^{-20}
\]
is reached. 

Figure~\ref{fig1} displays the evolution of the momentum norm $\|x^{n+1} - x^n\|$ versus the iteration count $n$, plotted on a logarithmic scale. 
We emphasize that, the momentum $\{x^{n+1}-x^n\}_{n\in\mathbb{N}}$ converges to zero at a $\rho$-linear rate if and only if the underlying iterates $\{x^n\}_{n\in\mathbb{N}}$ themselves converge at the same $\rho$-linear rate. 
Consequently, monitoring the decay of $\|x^{n+1} - x^n\|$ provides a reliable proxy for assessing the convergence behavior of the sequence.
\begin{figure}[ht]	
	\centering{	
		$k=10^{2}$\;\includegraphics[height=5cm,width=6cm]{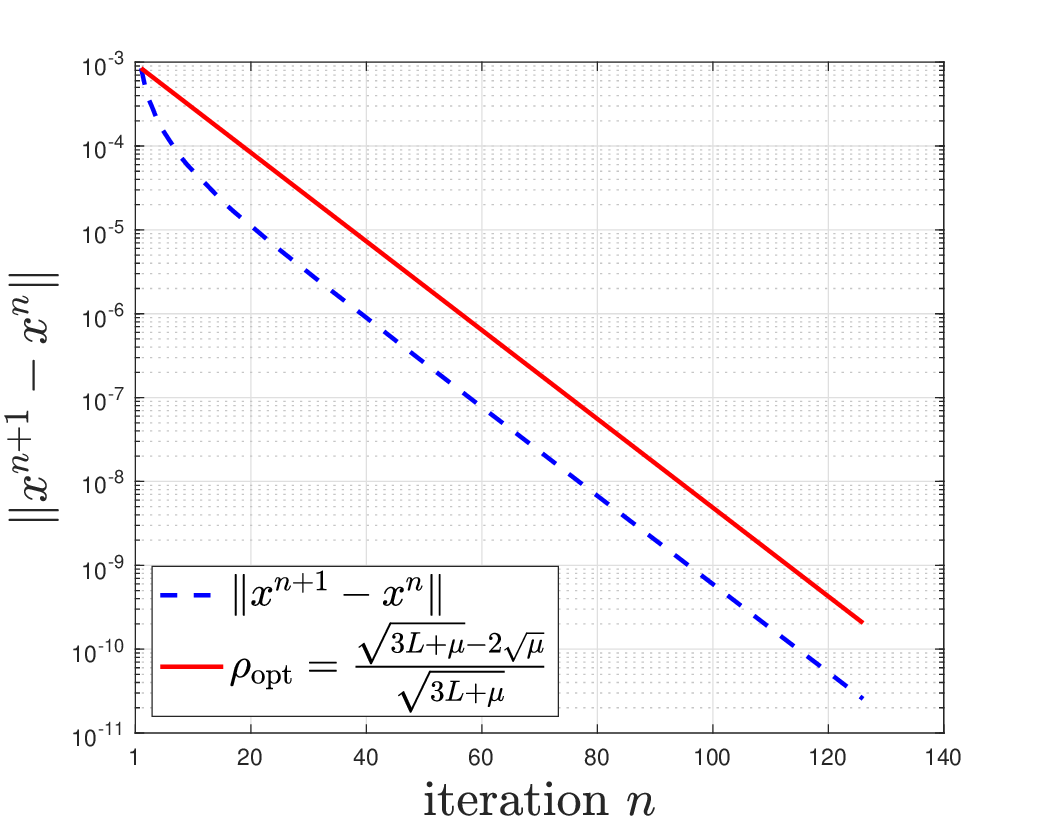} \quad
		$k=10^{3}$\;\includegraphics[height=5cm,width=6cm]{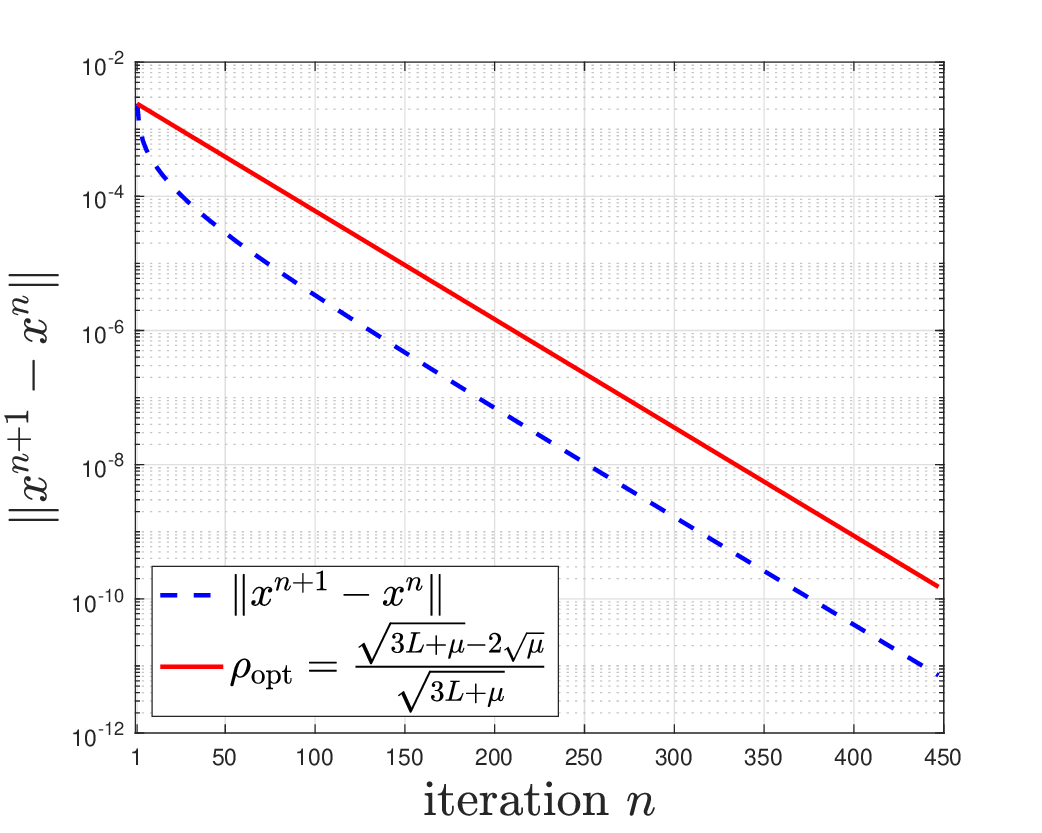}\\[1em]
		$k=10^{4}$\;\includegraphics[height=5cm,width=6cm]{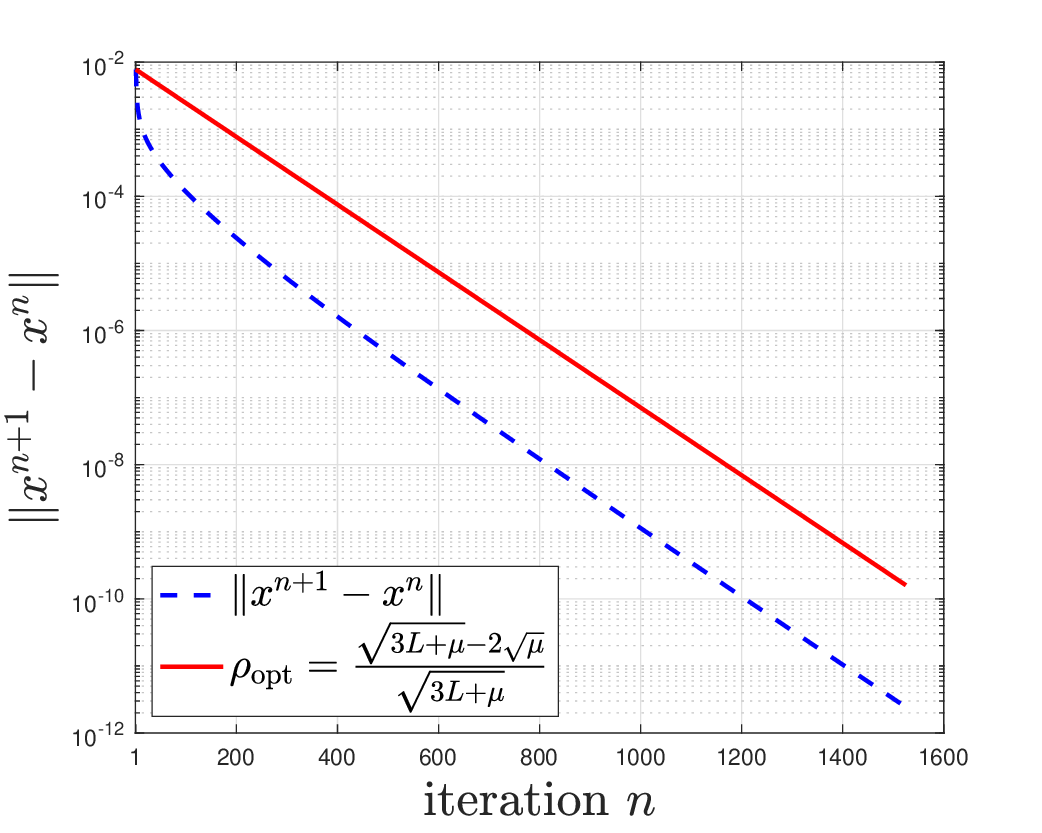}\quad
		$k=10^{5}$\;\includegraphics[height=5cm,width=6cm]{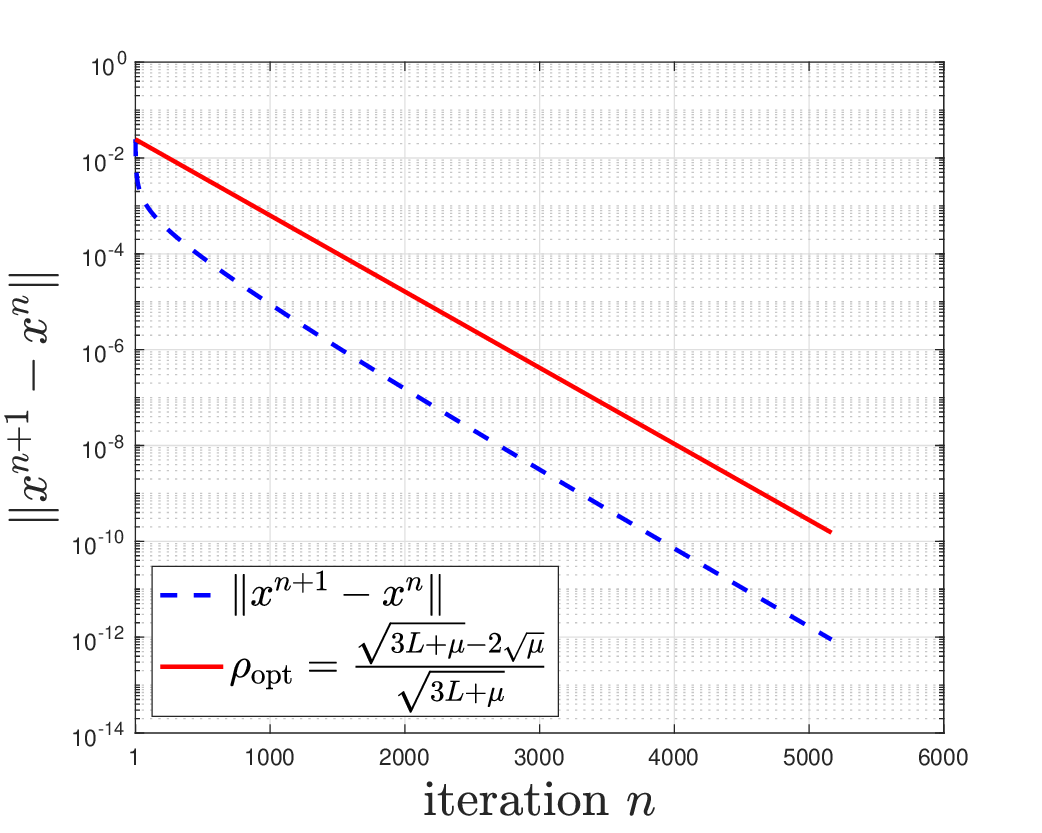}		
	}
	\caption{Convergence behavior of NAG for $k = 10^{2}$, $10^{3}$, $10^{4}$, and $10^{5}$. 
		In each subplot, the red solid line represents the theoretical optimal convergence rate $\rho_{\mathrm{opt}}$, 
		while the blue dashed line shows the actual error $\|x^{n+1} - x^{n}\|$.}
	\label{fig1}
\end{figure}

 As shown in the figure, the blue dashed curves—representing the empirical error $\|x^{n+1} - x^n\|$ across different problem scales ($k = 10^2, 10^3, 10^4, 10^5$)—asymptotically align with the red solid line, which corresponds to the theoretical optimal linear rate $\rho_{\mathrm{opt}}$. 
This alignment becomes increasingly pronounced in the later iterations, indicating that the algorithm enters its local linear convergence regime. In accordance with {\bf Theorem~\ref{Opt-convergence}}, both the theoretical prediction and numerical observation confirm a local $\rho_{\mathrm{opt}}$-linear decay, where the asymptotic higher-order terms (e.g., $\varepsilon$-level perturbations) have been omitted for clarity. 
Notably, the observed slope in the semi-log plot closely matches the analytically derived rate
\[
\rho_{\mathrm{opt}} = \frac{\sqrt{3L + \mu} - 2\sqrt{\mu}}{\sqrt{3L + \mu}},
\]
which is known to be the optimal local convergence factor for NAG in strongly convex settings. 

It is worth emphasizing that our experimental setting deliberately excludes the classical cases in which the local PL condition follows from (local) strong convexity. Instead, we consider genuinely nonconvex scenarios that are not locally strongly convex, yet still satisfy a local PL condition near the solution. This provides strong numerical evidence that NAG retains its optimal local acceleration rate beyond the traditional strongly convex framework—specifically, in regions where the PL inequality holds without requiring strong convexity. Consequently, the results corroborate both the robustness and sharpness of the guarantees established in Theorem~\ref{Opt-convergence}.

\section{Conclusion}\label{Sec5}

In this paper, we establish the optimal local linear convergence rate of Nesterov's Accelerated Gradient method (NAG) for solving locally non-strongly convex unconstrained optimization problems, where the objective function is $C^2$ and satisfies the local Polyak--\L{}ojasiewicz (PL) condition in a neighborhood of a local minimizer.  
This result naturally demonstrates that NAG retains its acceleration capability even in the absence of strong convexity within a local nonconvex regime.  Notably, under the assumption of the local PL condition, our analysis requires only a mild strengthening of the standard regularity assumption: we replace the $C^{1,1}$ condition (i.e., locally Lipschitz continuous gradient) commonly used with $C^2$ smoothness. This modest requirement renders our framework relatively mild and widely applicable. Finally, our numerical experiments are conducted on instances that deliberately exclude trivial cases such as locally strongly convex objectives. Specifically, we consider genuinely nonconvex and non-locally-strongly-convex functions that still satisfy the local PL condition near the solution. The close match between theory and practice in these representative scenarios strongly corroborates the validity, generality, and sharpness of our convergence guarantees.

\setcounter{equation}{0}


\begin{thebibliography}{10}
	
	\bibitem{2000Anitescu}
	Anitescu, M.: Degenerate nonlinear programming with a quadratic growth condition. SIAM J. Optim. \textbf{10}(4), 1116--1135 (2000)
	
	\bibitem{2024HBM}
	Aujol, J.-F., Dossal, C., Labari\'ere, H., Rondepierre, A.: Heavy ball momentum for non-strongly convex optimization. arXiv preprint arXiv:2403.06930 (2024)
	
	\bibitem{2025Criscitiello}
	Criscitiello, C., Rebjock, Q., Boumal, N.: If a smooth function is globally Polyak--\L{}ojasiewicz and coercive, then it has a unique minimizer. Race to the Bottom (2025) \url{https://www.racetothebottom.xyz/posts/PL-smooth-unique/}
	
	\bibitem{2025Gupta}
	Gupta, K., Wojtowytsch, S.: Nesterov acceleration in benignly non-convex landscapes. In: The International Conference on Learning Representations (2025) 
	
	\bibitem{2025SqC}
	Hermant, J., Aujol, J.-F., Dossal, C., Rondepierre, A.: Study of the behaviour of Nesterov accelerated gradient in a non-convex setting: the strongly quasar convex case. arXiv preprint arXiv:2405.19809 (2024)
	
	\bibitem{2026HJ}
	Hermant, J.: Acceleration for Polyak--\L{}ojasiewicz functions with a gradient aiming
	condition. arXiv preprint arXiv:2602.10022 (2026)
	
	\bibitem{2016Karimi}
	Karimi, H., Nutini, J., Schmidt, M.: Linear convergence of gradient and proximal-gradient methods under the Polyak--\L{}ojasiewicz condition. In: Joint European Conference on Machine Learning and Knowledge Discovery in Databases, pp~795--811. Springer (2016)
	
	\bibitem{2026HBOpt}
	Kassing, S. and Weissmann, S. Polyak’s heavy ball method achieves accelerated local rate of convergence under polyak--lojasiewicz inequality. arXiv preprint arXiv:2410.16849 (2024)
	
	\bibitem{lessard2016}
	L.~Lessard, B.~Recht, and A.~Packard, 
	Analysis and design of optimization algorithms via integral quadratic constraints, 
	SIAM J. Optim., 26 (2016), pp.~57--95.
	
	\bibitem{2014Liu}
	Liu, J., Wright, S., Re, C., Bittorf, V., Sridhar, S.: An asynchronous parallel stochastic coordinate descent algorithm. J. Mach. Learn. Res. \textbf{16}(1), 285--322 (2015)
	
	\bibitem{1963LS}
	\L{}ojasiewicz, S.: A topological property of real analytic subsets. Coll. du CNRS, Les
	\'equations aux d\'eriv\'ees partielles. \textbf{117}(2), 87--89 (1963)
	
	\bibitem{1993Luo}
	Luo, Z.-Q., Tseng, P.: Error bounds and convergence analysis of feasible descent methods: a general approach. Ann. Oper. Res. \textbf{46}(1), 157--178 (1993)
	
	\bibitem{2019Necoara}
	Necoara, I., Nesterov, Y., Glineur, F.: Linear convergence of first order methods for non-strongly convex optimization. Math. Program. \textbf{175}(1), 69--107 (2019)
	
	\bibitem{2025Nejma}
	Nejma, A.B.: Polyak--\L{}ojasiewicz inequality is essentially no more general than strong convexity for functions. arXiv preprint arXiv:2512.05285 (2025)
	
	\bibitem{2018Nestreov}
 	Nesterov, Y.: Lectures on Convex Optimization. Springer, Cham, 2nd ed (2018)
	
	\bibitem{1963Polyak}
	Polyak, B.T.: Gradient methods for the minimisation of functionals. USSR Comput. Math. and Math. Phys. \textbf{3}(4), 864--878 (1963)
	
	\bibitem{1964P}
	Polyak, B.T.: Some methods of speeding up the convergence of iteration methods. USSR Comput. Math. and Math. Phys., \textbf{4}(5), 1--17 (1964)
	
	\bibitem{2017Polyak}
	Polyak, B. T., Shcherbakov, P.: Lyapunov functions: An optimization theory perspective. IFACPapersOnLine, \textbf{50}(1), 7456--7461 (2017)
	
	\bibitem{2025Fast}
	Rebjock, Q., Boumal, N.: Fast convergence to non-isolated minima: four equivalent conditions for \(C^{2}\) functions. Math. Program. \textbf{213}, 151--199 (2025)
	
	\bibitem{2022Wang}
	Wang, J.-K., Lin, C.-H., Wibisono, A., Hu, B.: Provable acceleration of heavy ball beyond quadratics for a class of Polyak--\L{}ojasiewicz functions when the non-convexity is averaged out. In: Proceedings of the 39th International Conference on Machine Learning, pp~22839--22864. PMLR (2022)
	
	\bibitem{2023OnLB}
	Yue, P., Fang, C., Lin, Z.: On the lower bound of minimizing Polyak--Łojasiewicz functions. In: The Thirty Sixth Annual Conference on Learning Theory, pp~2948–2968. PMLR (2023)
	
	\bibitem{2013Zhang}
	Zhang, H., Yin, W.: Gradient methods for convex minimization: better rates under weaker conditions. arXiv preprint arXiv:1303.4645 (2013)
	
\end{thebibliography}
\end{document}